\documentclass[11pt,reqno]{amsart}

\usepackage{amsmath,amsthm}
\usepackage{amssymb,amsfonts}
\usepackage{latexsym}
\usepackage{mathrsfs}
\usepackage{bm}
\usepackage{yhmath}
\usepackage{tikz}

\theoremstyle{plain}
\newtheorem{theorem}{Theorem}[section]
\newtheorem{lemma}[theorem]{Lemma}

\newtheorem*{maintheorem}{Theorem \ref{thm:main theorem}}

\theoremstyle{definition}
\newtheorem{definition}[theorem]{Definition}
\newtheorem*{notation}{Notation}

\theoremstyle{remark}
\newtheorem{remark}{Remark}

\numberwithin{equation}{section}

\allowdisplaybreaks

\def\XXint#1#2#3{{\setbox0=\hbox{$#1{#2#3}{\int}$}
    \vcenter{\hbox{$#2#3$}}\kern-.5\wd0}}

\makeatletter
\def\@citestyle{\m@th\upshape\mdseries}
\def\citeform#1{{\bfseries#1}}
\def\@cite#1#2{{%
  \@citestyle[\citeform{#1}\if@tempswa, #2\fi]}}
\@ifundefined{cite }{%
  \expandafter\let\csname cite \endcsname\cite
  \edef\cite{\@nx\protect\@xp\@nx\csname cite \endcsname}%
}{}
\makeatother

\renewcommand{\leq}{\leqslant}
\renewcommand{\geq}{\geqslant}

\newcommand{\inner}[2]{\langle #1\,,#2\rangle}

\renewcommand{\Im}{\mathop{\mathrm{Im}}}
\renewcommand{\Re}{\mathop{\mathrm{Re}}}

\renewcommand{\span}{\mathop{\mathrm{span}}}

\newcommand{\C}{\mathbb{C}}
\newcommand{\R}{\mathbb{R}}

\newcommand{\B}{\mathbb{B}}
\renewcommand{\H}{\mathbb{H}}

\newcommand{\Lbb}{\mathbb{L}}

\newcommand{\HH}{\mathbf{H}}

\DeclareMathOperator{\I}{I}%
\DeclareMathOperator{\II}{II}%
\DeclareMathOperator{\Isom}{Isom}%
\DeclareMathOperator{\Ric}{Ric}%
\DeclareMathOperator{\vol}{vol}%

\DeclareMathOperator{\sech}{sech}%

\newcommand{\CC}{\mathcal{C}}
\newcommand{\Hcal}{\mathcal{H}}
\newcommand{\Lcal}{\mathcal{L}}

\newcommand{\Mbar}{\overline{M}}
\newcommand{\Mscr}{\mathscr{M}}

\renewcommand{\O}{\mathsf{O}}%
\newcommand{\SO}{\mathsf{SO}}%

\newcommand{\abar}{\bar{a}}
\newcommand{\atilde}{\tilde{a}}

\title[Stability of Helicoids in $\H^3$]
{Stability of Helicoids in Hyperbolic
Three-Dimensional Space}

\begin{document}

\author{Biao Wang}
\date{\today}

\address{Department of Mathematics and Computer Science\\
         Queensborough College, The City University of New York\\
         222-05 56th Avenue Bayside, NY 11364\\}
\email{biwang@qcc.cuny.edu}

\begin{abstract}
  For a family of minimal helicoids $\{\Hcal_a\}_{a\geq{}0}$
  in the hyperbolic $3$-space $\H^3$ (see $\S$\ref{sec:helicoid-I}
  or $\S$\ref{sec:helicoid-II} for detail definitions),
  there exists a constant $a_c\approx{}2.17966$ such that
  the following statements are true:
  \begin{itemize}
     \item $\Hcal_a$ is a globally stable minimal surface if
           $0\leq{}a\leq{}a_c$, and
     \item $\Hcal_a$ is an unstable minimal surface with index one
           if $a>a_c$.
  \end{itemize}
\end{abstract}

\maketitle

\section{Introduction}\label{sec:Introduction}

Let $(\Mbar{}^{3},\overline{g})$ be a $3$-dimensional Riemannian
manifold (compact or complete), and let $\Sigma$ be a
surface (compact or complete) immersed in $\Mbar{}^{3}$. We choose a
local orthonormal frame field $\{e_1,e_2,e_3\}$ for
$\Mbar{}^{3}$ such that, restricted to $\Sigma$, the vectors
$\{e_1,e_2\}$ are tangent to $\Sigma$ and the vector $e_3$
is perpendicular to $\Sigma$. Let $A=(h_{ij})_{2\times{}2}$
denote the second fundamental
form of $\Sigma$, whose entries $h_{ij}$ are represented by
\begin{equation*}
   h_{ij}=\inner{\overline{\nabla}_{e_i}e_{3}}{e_{j}}\ ,
   \quad{}i,j=1,2\ ,
\end{equation*}
where $\overline{\nabla}$ is the covariant derivative in
$\Mbar{}^{3}$, and
$\inner{\cdot}{\cdot}$ is the metric of $\Mbar{}^{3}$.
The immersed surface
$\Sigma$ is called a \emph{minimal surface} if its
\emph{mean curvature} $H=h_{11}+h_{22}$ is identically equal
to zero.

Let $\overline{\Ric}(e_3)$ denote the Ricci curvature of $\Mbar{}^{3}$
in the direction $e_3$, and $|A|^2=\sum_{i,j=1}^{2}h_{ij}^2$.
The \emph{Jacobi operator} on $\Sigma$ is defined by
\begin{equation}\label{eq:Jacobi operator I}
   \Lcal=\Delta_{\Sigma}+(|A|^2+\overline{\Ric}(e_3))\ ,
\end{equation}
where $\Delta_{\Sigma}$ is the Lapalican for the induced
metric on $\Sigma$.

\subsection{Stability of minimal surfaces}
Suppose that $\Sigma$ is a complete minimal surface immersed in
a complete Riemannian manifold $\Mbar{}^{3}$.
For any compact connected subdomain $\Omega$ of $\Sigma$,
its \emph{first eigenvalue} is defined by
\begin{equation}\label{eq:1st eigenvalue of Omega}
   \lambda_{1}(\Omega)=\inf\left\{-\int_{\Omega}u\Lcal{}u\,d\vol
   \ \left|\ u\in{}C_{0}^\infty(\Omega)\ \text{and}\
   \int_{\Omega}u^{2}\,d\vol=1\right.\right\}\ ,
\end{equation}
where $d\vol$ denotes the area element on the surface $\Sigma$.
We say that $\Omega$ is \emph{stable} if $\lambda_{1}(\Omega)>0$,
\emph{unstable} if $\lambda_{1}(\Omega)<0$ and
\emph{maximally weakly stable} if $\lambda_{1}(\Omega)=0$.
For a complete minimal surface $\Sigma$ without boundary, it is
\emph{globally stable} or \emph{stable} if any compact subdomain
of $\Sigma$ is stable.

\begin{remark}It's easy to verify that $\Omega\subset\Sigma$ is
stable if and only if
\begin{equation}
   \int_{\Omega}|\nabla_{\Sigma}u|^{2}d\vol>
   \int_{\Omega}(|A|^2+\overline{\Ric}(e_3))u^{2}d\vol
\end{equation}
for all $u\in{}C_{0}^{\infty}(\Omega)$, where $\nabla_{\Sigma}$ is the
covariant derivative of $\Sigma$.
\end{remark}

The following result is well known (for example see
\cite[Lemma 6.2.5]{Xin03}).

\begin{lemma}\label{lem:monotonicity of eigenvalue}
Let $\Omega_1$ and $\Omega_2$ be connected subdomains of $\Sigma$
with $\Omega_1\subset\Omega_2$, then
\begin{equation*}
  \lambda_{1}(\Omega_1)\geq\lambda_{1}(\Omega_2)\ .
\end{equation*}
If $\Omega_2\setminus\overline{\Omega}_1\ne\emptyset$, then
\begin{equation*}
  \lambda_{1}(\Omega_1)>\lambda_{1}(\Omega_2)\ .
\end{equation*}
\end{lemma}

\begin{remark}If $\Omega\subset\Sigma$ is maximally weakly
stable, then for any compact connected subdomains
$\Omega_1,\Omega_2\subset\Sigma$ satisfying
$\Omega_1\subsetneq\Omega\subsetneq\Omega_2$, we have that
$\Omega_1$ is stable whereas $\Omega_2$ is unstable.
\end{remark}

Let
$\Omega_1\subset\Omega_2\subset\cdots\subset\Omega_n\subset\cdots$
be an exhaustion of $\Sigma$, then the first eigenvalue of
$\Sigma$ is defined by
\begin{equation}\label{eq:1st eigenvalue of Sigma}
   \lambda_{1}(\Sigma)=\lim_{n\to\infty}\lambda_{1}(\Omega_n)\ .
\end{equation}
This definition is independent of the choice of the exhaustion.
We say that $\Sigma$ is \emph{stable} if $\lambda_{1}(\Sigma)>0$ and
\emph{unstable} if $\lambda_{1}(\Sigma)<0$.

Let $\Sigma\subset{}\Mbar{}^{3}$ be a complete minimal surface, and let
$\Omega$ be any subdomain of $\Sigma$. Recall that a
\emph{Jacobi field} on $\Omega$ is a $C^\infty$ function
$\phi$ such that $\Lcal\phi = 0$ on $\Omega$.
The following theorem was proved by Fischer-Colbrie and Schoen in
\cite[Theorem 1]{FCS80} (see also \cite[Proposition 1.39]{CM11}).

\begin{theorem}[Fischer-Colbrie and Schoen]\label{thm:FCS80}
If $\Sigma$ is a complete minimal surface in
a $3$-dimensional Riemannian manifold $\Mbar{}^{3}$, then $\Sigma$ is
stable if and only if there exists a positive function
$\phi:\Sigma\to\R$ such that $\Lcal{}\phi=0$.
\end{theorem}

The \emph{Morse index} or \emph{index} of compact connected subdomain
$\Omega$ of $\Sigma$ is the number of negative eigenvalues of the
Jacobi operator $\Lcal$ (counting with multiplicity) acting on the
space of smooth sections of the normal bundle that vanishes on
$\partial\Omega$. The \emph{Morse index} of $\Sigma$ is the supremum of
the Morse indices of compact subdomains of $\Sigma$.


\subsection{Helicoids in $\HH^3$ and the main theorem}\label{sec:helicoid-I}
We consider the Lorentzian $4$-space
$\Lbb^{4}$, i.e. a vector space $\R^{4}$ with the Lorentzian
inner product
\begin{equation}\label{eq:Lorentzian inner product}
   \inner{x}{y}=-x_{1}y_{1}+x_{2}y_{2}+x_{3}y_{3}+x_{4}y_{4}
\end{equation}
where $x,y\in\R^{4}$. The hyperbolic space $\HH^{3}$ can be
considered as the unit sphere of $\Lbb^{4}$:
\begin{equation}
   \HH^{3}=\{x\in\Lbb^{4}\ |\ \inner{x}{x}=-1,\,x_{1}\geq{}1\}\ .
\end{equation}
The helicoid $\Hcal_{a}$ is the surface parametrized by the
$(u,v)$-plane in the following way:
\begin{equation}\label{eq:helicoid in hyperboloid model}
   \Hcal_{a}=\left\{x\in\HH^{3}\ \left|
   \begin{aligned}
      &x_{1}=\cosh{}u\cosh{}v, &&x_{2}=\cosh{}u\sinh{}v\\
      &x_{3}=\sinh{}u\cos(av), &&x_{4}=\sinh{}u\sin(av)
   \end{aligned}
   \right.\right\}
\end{equation}
For any constant $a\geq{}0$, the helicoid $\Hcal_a\subset\HH^3$ is an
embedded minimal surface (see \cite{Mor82}). In the hyperboloid model
$\HH^3$, the axis of the helicoid $\Hcal_{a}$ is given by
\begin{equation*}
   (\cosh{}v,\sinh{}v,0,0)\ ,
   \quad{}-\infty<v<\infty\ ,
\end{equation*}
which is the intersection of the $x_{1}x_{2}$-plane and the
three dimensional hyperboloid $\HH^3$.

In this paper, we will prove the following theorem.

\begin{maintheorem}
For a family of minimal helicoids $\{\Hcal_a\}_{a\geq{}0}$
in the hyperbolic $3$-space $\HH^3$ given by
\eqref{eq:helicoid in hyperboloid model},
there exist a constant $a_c\approx{}2.17966$ such that
the following statements are true:
  \begin{enumerate}
    \item $\Hcal_a$ is a globally stable minimal surface
          if $0\leq{}a\leq{}a_c$, and
    \item $\Hcal_a$ is an unstable minimal surface with index one
          if $a>a_c$.
  \end{enumerate}
\end{maintheorem}

\begin{remark}In \cite[Theorem 2]{Mor82}, Mori proved that
$\Hcal_a$ is globally stable if
$a\leq{}3\sqrt{2}/4\approx{}1.06$ (see also
\cite[Theorem 5.1]{Seo11}), and
that $\Hcal_a$ is unstable if
$a\geq{}\sqrt{105\pi}/8\approx{}2.27$.
\end{remark}

\begin{notation}We use the following notations for
the $3$-dimensional hyperbolic space in this paper:
(1) $\HH^{3}$ denotes the hyperboloid model;
(2) $\B^{3}$ denotes the Pinecar{\'e} ball model; and
(3) $\H^{3}$ denotes the upper half space model.
Usually we use $\H^{3}$ to denote the hyperbolic $3$-space for the
most time.
\end{notation}

\section{Helicoids in hyperbolic space}\label{sec:helicoid-II}

In order to visualize the helicoids in the hyperbolic $3$-space,
we will study the parametric equations of helicoids in the
Poincar{\'e} ball model and upper half space model of
the hyperbolic $3$-space.

Consider the Pinecar{\'e} ball model $\B^{3}$, i.e., the unit sphere
\begin{equation*}
   \B^{3}=\{(x,y,z)\in\R^{3}\ | \ x^{2}+y^{2}+z^{2}<{}1\},
\end{equation*}
equipped with the hyperbolic metric
\begin{equation*}
   ds^{2}=\frac{4(dx^{2}+dy^{2}+dz^{2})}{(1-r^{2})^{2}}\ ,
\end{equation*}
where $r=\sqrt{x^{2}+y^{2}+z^{2}}$.
The helicoid $\Hcal_a$ in the ball model $\B^3$ is given by
\begin{equation}\label{eq:helicoid--ball model}
  \Hcal_{a}=\left\{(x,y,z)\in\B^{3}\ \left|\
     \begin{aligned}
        x&=\frac{\sinh{}u{}\cos(av)}{1+\cosh{}u{}\cosh{}v}\ , \\
        y&=\frac{\sinh{}u{}\sin(av)}{1+\cosh{}u{}\cosh{}v}\ , \\
        z&=\frac{\cosh{}u{}\sinh{}v}{1+\cosh{}u{}\cosh{}v}
     \end{aligned}\right.\right\}\ ,
\end{equation}
where $-\infty<u,v<\infty$.

\begin{figure}[htbp]
  \centering
  \includegraphics[scale=0.2]{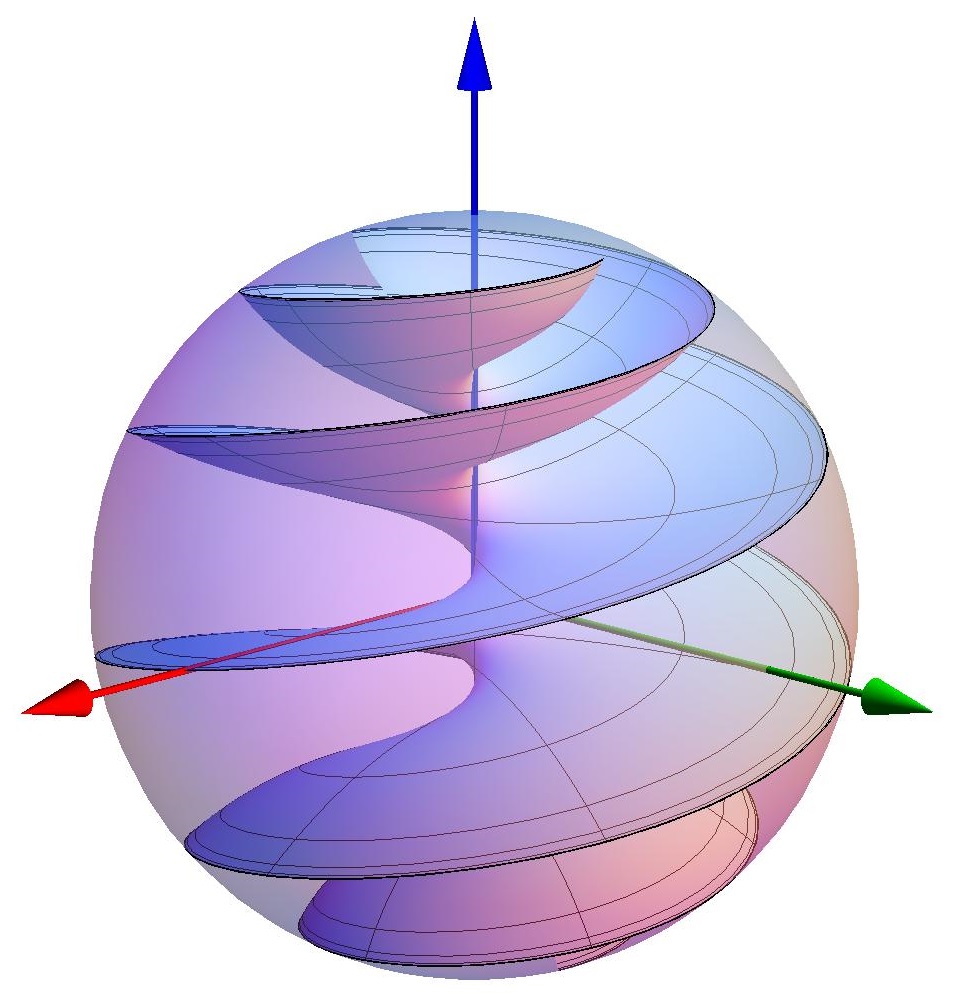}
  \caption{The helicoid $\Hcal_a$ with $a=5$ in the Poincar{\'e} ball
  model of hyperbolic space. The curves perpendicular
  to the spirals are geodesics in $\B^3$.}\label{fig:helicoid in B3}
\end{figure}

Consider the upper half space model of hyperbolic $3$-space,
i.e., a three dimensional space
\begin{equation*}
   \H^{3}=\{z+tj\ |\ z\in\C\ \text{and}\ t>0\}\ ,
\end{equation*}
which is equipped with the (hyperbolic) metric
\begin{equation*}
   ds^{2}=\frac{|dz|^{2}+dt^{2}}{t^{2}}\ ,
\end{equation*}
where $z=x+iy$ for $x,y\in\R$. In the upper half space model,
the helicoid $\Hcal_a$ is given by
(see Figure \ref{fig:helicoid in H3})
\begin{equation}\label{eq:helicoid--upper half}
  \Hcal_{a}=\{(z,t)\in\H^{3}\ |\ z=e^{v+iav}\tanh{}u\
     \text{and}\ t=e^{v}\sech{}u\}\ ,
\end{equation}
and the axis of $\Hcal_a$ is the $t$-axis, where $-\infty<u,v<\infty$.

\begin{figure}[htbp]
  \centering
  \includegraphics[scale=0.2]{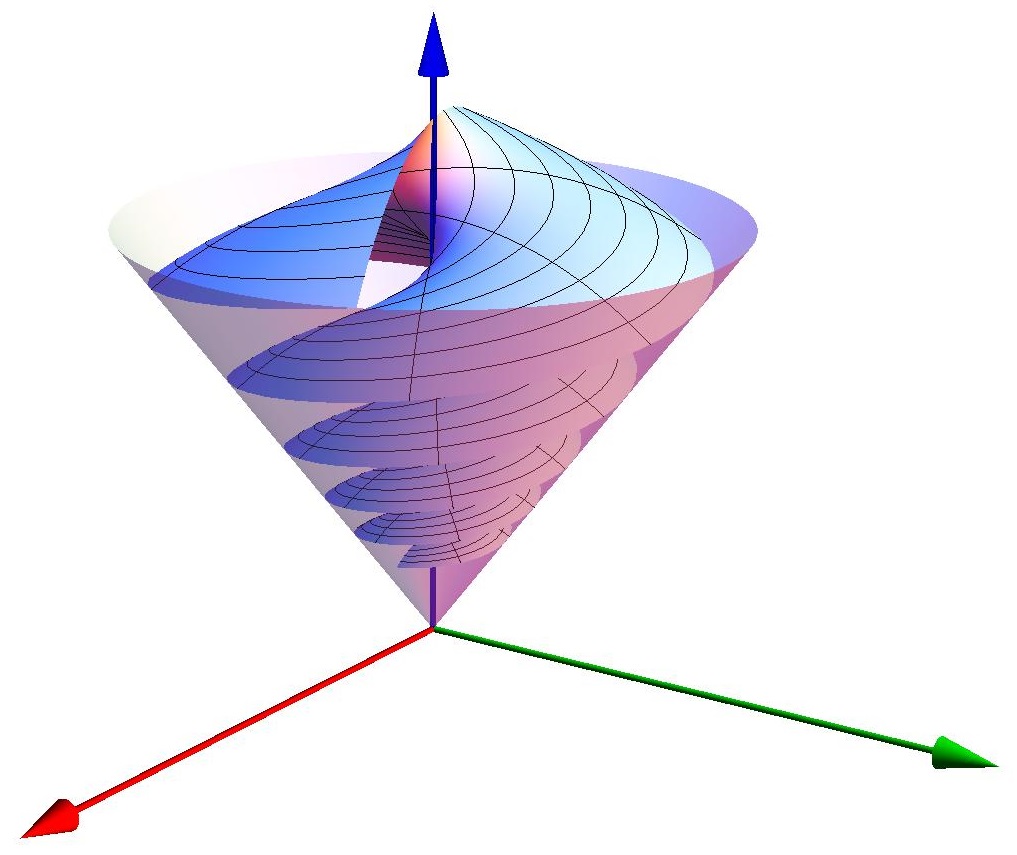}
  \caption{The helicoid $\Hcal_a$ with $a=10$ in the upper half space
  model of hyperbolic space, where the cone whose vertex is the origin
  is the $\log{}2$-neighborhood of the $t$-axis. The curves perpendicular
  to the spirals are geodesics in $\H^3$.}\label{fig:helicoid in H3}
\end{figure}

\begin{remark}
To derive the formulas in \eqref{eq:helicoid--ball model}
and \eqref{eq:helicoid--upper half},
we apply the isometries from the hyperboloid model to the
Pinecar{\'e} ball model and the upper half space model
(see \cite[$\S$A.1]{BP92}).
\end{remark}

\section{Catenoids in hyperbolic $3$-space}\label{sec:catenoids}

In this section we define the catenoids (rotation surfaces with
zero mean curvature) in hyperbolic $3$-space $\H^3$ (see
\cite{dCD83,Wang2012} for detail), since we need these to prove
Theorem~\ref{thm:main theorem}.

\subsection{Catenoids in the hyperboloid model $\HH^3$}
\label{subsec:catenoids in hyperboloid model}

We follow \cite{dCD83} to describe three types of catenoids in
$\HH^3$.
Recall that $\Lbb^4$ is the Lorentzian $4$-space whose inner product is
given by \eqref{eq:Lorentzian inner product}. Its isometry group is
$\SO^{+}(1,3)$. For any subspace $P$ of $\Lbb^4$, let $\O(P)$ be
the subgroup of $\SO^{+}(1,3)$ which leaves $P$ pointwise fixed.

\begin{definition}\label{def:rotation surface}
Let $\{e_1,\ldots,e_4\}$ be an orthonormal basis of
$\Lbb^4$ (it may not be the standard orthonormal basis). Suppose
that $P^2=\span\{e_3,e_4\}$, $P^3=\span\{e_1,e_3,e_4\}$ and
$P^3\cap\HH^3\ne\emptyset$. Let $C$ be a regular curve in
$P^3\cap\HH^3=\HH^2$ that does not meet $P^2$.
The orbit of $C$ under the action of $\O(P^2)$ is called a
\emph{rotation surface} \emph{generated by $C$ around $P^2$}.
\end{definition}

If a rotation surface in Definition \ref{def:rotation surface}
has mean curvature zero, then it's called a \emph{catenoid} in
$\HH^3$. There are three types of catenoids
in $\HH^3$: spherical catenoids, hyperbolic catenoids, and
parabolic catenoids.

\subsubsection{Spherical catenoids}
\label{subsub:spherical catenoild-lorentz}
The spherical catenoid is obtained as follows.
Let $\{e_1,\ldots,e_4\}$ be an orthonormal basis of $\Lbb^4$ such
that $\inner{e_4}{e_4}=-1$. Suppose that $P^2$, $P^3$ and $C$ are the
same as those defined in Definition \ref{def:rotation surface}.
For any point $x\in\Lbb^4$, write $x=\sum{}x_{k}e_{k}$. If
the curve $C$ is parametrized by
\begin{equation}\label{eq:spherical catenoid-x1(s)}
   x_{1}(s)=\sqrt{\atilde\cosh(2s)-1/2}\ ,
   \quad \atilde>1/2\,
\end{equation}
and
\begin{equation}\label{eq:spherical catenoid-x3(s)-x4(s)}
   x_{3}(s)=\sqrt{x_{1}^{2}(s)+1}\,\sinh(\phi(s))\ ,\
   x_{4}(s)=\sqrt{x_{1}^{2}(s)+1}\,\cosh(\phi(s))\ ,
\end{equation}
where
\begin{equation}
   \phi(s)=\int_{0}^{s}\frac{\sqrt{\atilde{}^2-1/4}}
   {(\atilde\cosh(2\sigma)+1/2)
   \sqrt{\atilde\cosh(2\sigma)-1/2}}\,d\sigma\ ,
\end{equation}
then the rotation surface, denoted by $\Mscr_{1}^2(\atilde)$,
is a complete minimal surface in $\HH^3$,
which is called a \emph{spherical catenoid}.

\subsubsection{Hyperbolic catenoids}
\label{subsub:hyperbolic catenoild-lorentz}
The hyperbolic catenoid is obtained as follows.
Let $\{e_1,\ldots,e_4\}$ be an orthonormal basis of $\Lbb^4$
such that $\inner{e_1}{e_1}=-1$. Suppose that $P^2$, $P^3$ and $C$ are the
same as those defined in Definition \ref{def:rotation surface}.
For any point $x\in\Lbb^4$, write $x=\sum{}x_{k}e_{k}$. If
the curve $C$ is parametrized by
\begin{equation}\label{eq:hyperbolic catenoid-x1(s)}
   x_{1}(s)=\sqrt{\atilde\cosh(2s)+1/2}\ ,
   \quad \atilde>1/2\,
\end{equation}
and
\begin{equation}\label{eq:hyperbolic catenoid-x3(s)-x4(s)}
   x_{3}(s)=\sqrt{x_{1}^{2}(s)-1}\,\sin(\phi(s))\ ,\
   x_{4}(s)=\sqrt{x_{1}^{2}(s)-1}\,\cos(\phi(s))\ ,
\end{equation}
where
\begin{equation}
   \phi(s)=\int_{0}^{s}\frac{\sqrt{\atilde{}^2-1/4}}
   {(\atilde\cosh(2\sigma)-1/2)
   \sqrt{\atilde\cosh(2\sigma)+1/2}}\,d\sigma\ ,
\end{equation}
then the rotation surface, denoted by $\Mscr_{-1}^2(\atilde)$,
is a complete minimal surface in $\HH^3$,
which is called a \emph{hyperbolic catenoid}.

\subsubsection{Parabolic catenoids}
\label{subsub:parabolic catenoild-lorentz}
The parabolic catenoid is obtained as follows.
Let $\{e_1,\ldots,e_4\}$ be a pseudo-orthonormal basis of $\Lbb^4$ such
that $\inner{e_1}{e_1}=\inner{e_3}{e_3}=0$, $\inner{e_1}{e_3}=-1$ and
$\inner{e_j}{e_k}=\delta_{jk}$ for $j=2,4$ and $k=1,2,3,4$
(see \cite[P. 689]{dCD83}).
Suppose that $P^2$, $P^3$ and $C$ are the
same as those defined in Definition \ref{def:rotation surface}.
For any point $x\in\Lbb^4$, write $x=\sum{}x_{k}e_{k}$. If
the curve $C$ is parametrized by
\begin{equation}\label{eq:parabolic catenoid-x1(s)}
   x_{1}(s)=\sqrt{\cosh(2s)}\ ,
\end{equation}
and
\begin{equation}\label{eq:parabolic catenoid-x3(s)-x4(s)}
   x_{4}(s)=x_{1}(s)\int_{0}^{s}\frac{d\sigma}
   {\sqrt{\cosh^3(2\sigma)}}\ ,\
   x_{3}(s)=\frac{1+x_{4}^2(s)}{x_{1}(s)}\ ,
\end{equation}
then the rotation surface, denoted by $\Mscr_{0}^2$,
is a complete minimal surface in $\HH^3$, which is called a
\emph{parabolic catenoid}.
Up to isometries, the parabolic catenoid $\Mscr_{0}^2$ is unique
(see \cite[Theorem (3.14)]{dCD83}).

\subsection{Catenoids in the Poincar{\'e} ball model $\B^3$}
\label{subsec:catenoids in Poincare model}
We follow \cite{Wang2012} to describe the spherical catenoids in
$\B^3$. Suppose that $G$ is a subgroup of $\Isom^{+}(\B^{3})$ that
leaves a geodesic $\gamma\subset\B^{3}$ pointwise fixed. We call $G$ the
\emph{spherical group} of $\B^{3}$ and $\gamma$ the \emph{rotation axis}
of $G$. A surface in $\B^{3}$ invariant under $G$ is called a
\emph{spherical surface} or a \emph{surface of revolution}.

Now suppose that $G$ is the spherical group of $\B^{3}$ along the geodesic
\begin{equation}\label{eq:rotation axis}
   \gamma_{0}=\{(u,0,0)\in\B^3\ |\ -1<u<1\}\ ,
\end{equation}
then $\B^3/G\cong\B_{+}^2$, where
$\B_{+}^2=\{(u,v)\in\B^2\ |\ v\geq{}0\}$.

For any point $p=(u,v)\in\B_{+}^2$, there is a unique geodesic segment
$\gamma'$ passing through $p$ that is perpendicular to $\gamma_{0}$ at $q$.
Let $x=d(O,q)$ and $y=d(p,q)=d(p,\gamma_{0})$
(see Figure~\ref{fig:intrinsic metric}), where
$d(\cdot,\cdot)$ denotes the hyperbolic distance.
It's well known that $\B_{+}^2$ can be equipped with
the \emph{metric of warped product} in terms
of the parameters $x$ and $y$ as follows:
\begin{equation}\label{eq:warped product metric}
   ds^2=\cosh^{2}y\cdot{}dx^2+dy^2\ ,
\end{equation}
where $dx$ represents the hyperbolic metric on the geodesic
$\gamma_0$ in \eqref{eq:rotation axis}.

\begin{figure}[htbp]
  \begin{center}
     \includegraphics[scale=1]{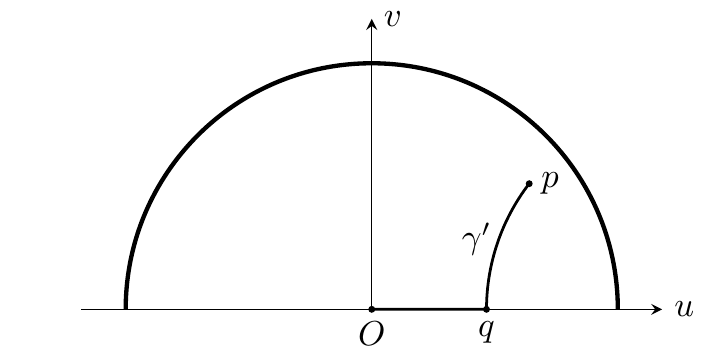}
   \end{center}
  \caption{For a point $p$ in $\B_{+}^2$ with the warped product metric,
  its coordinates $(x,y)$ are defined by $x=d(O,q)$ and
  $y=d(p,q)$.}\label{fig:intrinsic metric}
\end{figure}

Suppose that $\CC$ is a surface of revolution in $\B^3$ with respect to
the geodesic $\gamma_{0}$, then the curve $\sigma=\CC\cap\B_{+}^{2}$ is
called the \emph{generating curve} of $\CC$. If the curve
$\sigma_{\abar}\subset\B_{+}^{2}$ is given by the parametric equations:
\begin{equation*}
   x(t)=\pm\int_{\abar}^{t}\frac{\sinh(2\abar)}{\cosh{}\tau}
           \frac{d\tau}{\sqrt{\sinh^2(2\tau)-\sinh^2(2\abar)}}\ ,
\end{equation*}
and $y(t)=t$, where $\abar>0$ is a constant and $t\geq{}\abar$, then the
spherical surface generated by $\sigma_{\abar}$, denoted by $\CC_{\abar}$,
is a complete minimal surface in $\B^3$, which is called a
\emph{spherical catenoid}.

The following result
will be used for proving Theorem \ref{thm:main theorem}.

\begin{lemma}[B{\'e}rard and Sa Earp]
\label{lem:relation between spherical catenoids}
The spherical catenoid $\Mscr_{1}^{2}(\atilde)$ defined
in $\S${}\ref{subsub:spherical catenoild-lorentz} is isometric to the
spherical catenoid $\CC_{\abar}$, where
\begin{equation}\label{eq:relation between catenoids}
   2\atilde=\cosh(2\abar)\ .
\end{equation}
\end{lemma}

\begin{proof}The spherical catenoid $\CC_{\abar}$ is
obtained by rotating the generating curve
$\sigma_{\abar}$ along the axis $\gamma_{0}$.
The distance
between $\sigma_{\abar}$ and $\gamma_{0}$ is $\abar$.

The spherical catenoid $\Mscr_{1}^{2}(\atilde)$
defined in $\S${}\ref{subsub:spherical catenoild-lorentz}
can be obtained by rotating the generating
curve $C$ given by \eqref{eq:spherical catenoid-x1(s)}
and \eqref{eq:spherical catenoid-x3(s)-x4(s)} along the
geodesic $P^2\cap\HH^{3}$. The distance
between $C$ and $P^2\cap\HH^{3}$ is
$\sinh^{-1}\left(\sqrt{\atilde-1/2}\,\right)$.

Hence the spherical catenoid $\Mscr_{1}^{2}(\atilde)$ is
isometric to the spherical catenoid $\CC_{\abar}$ if and only if
$\abar=\sinh^{-1}\left(\sqrt{\atilde-1/2}\,\right)$, which
implies \eqref{eq:relation between catenoids}.
\end{proof}

\begin{remark}The equation \eqref{eq:relation between catenoids}
can be found in \cite[p. 34]{BSE09}.
\end{remark}

\section{Conjugate minimal surface}

Let $\Mbar{}^{3}(c)$ be the $3$-dimensional space form whose
sectional curvature is a constant $c$.

\begin{definition}[{\cite[pp. 699-700]{dCD83}}]
\label{def:conjugate minimal immersion}
Let $f:\Sigma\to\Mbar{}^{3}(c)$ be a minimal surface in isothermal
parametrs $(\sigma,t)$. Denote by
\begin{equation*}
   \I=E(d\sigma^2+dt^2)
   \quad\text{and}\quad
   \II=\beta_{11}d\sigma^2+2\beta_{12}d\sigma{}dt+
   \beta_{22}dt^2
\end{equation*}
the first and second fundamental forms of $f$, respectively.

Set $\psi=\beta_{11}-i\beta_{12}$ and define a family of
quadratic form depending on a parameter $\theta$,
$0\leq\theta\leq{}2\pi$, by
\begin{equation}
   \beta_{11}(\theta)=\Re(e^{i\theta}\psi)\ ,
   \quad
   \beta_{22}(\theta)=-\Re(e^{i\theta}\psi)\ ,
   \quad
   \beta_{12}(\theta)=\Im(e^{i\theta}\psi)\ .
\end{equation}
Then the following forms
\begin{equation*}
   \I_{\theta}=\I
   \quad\text{and}\quad
   \II_{\theta}=\beta_{11}(\theta)d\sigma^2+
   2\beta_{12}(\theta)d\sigma{}dt+
   \beta_{22}(\theta)dt^2
\end{equation*}
give rise to an isometry family
$f_{\theta}:\widetilde\Sigma\to\Mbar{}^{3}(c)$ of minimal
immersions, where $\widetilde\Sigma$ is the universal covering
of $\Sigma$. The immersion $f_{\pi/2}$ is called the
\emph{conjugate immersion} to $f_0=f$.
\end{definition}

The following result is obvious, but it's crucial to prove
Theorem \ref{thm:main theorem}.

\begin{lemma}Let $f:\Sigma\to\Mbar{}^{3}(c)$ be an immersed
minimal surface, and let
$f_{\pi/2}:\widetilde\Sigma\to\Mbar{}^{3}(c)$ be its conjugate
minimal surface, where $\widetilde\Sigma$ is the universal
covering of $\Sigma$. Then the minimal immersion $f$ is
globally stable if and only if its conjugate immersion
$f_{\pi/2}$ is globally stable.
\end{lemma}

\begin{proof}Let $\widetilde{f}$ be the universal lifting
of $f$, then
$\widetilde{f}:\widetilde\Sigma\to\Mbar{}^{3}(c)$ is a minimal
immersion. It's well known that the global stability of $f$
implies the global stability of $\widetilde{f}$.
Actually the minimal surfaces $\Sigma$ and $\widetilde\Sigma$
share the same Jacobi operator defined by
\eqref{eq:Jacobi operator I}. If $\Sigma$ is globally stable,
there exists a positive Jacobi field on $\Sigma$ according to
Theorem \ref{thm:FCS80}, which implies that $\widetilde\Sigma$
is also globally stable, since the corresponding positive Jacobi
field on $\widetilde\Sigma$ is given by composing.

Next we claim that $\widetilde{f}$ and $f_{\pi/2}$ share the
same Jacobi operator. In fact, since the Laplacian depends only
on the first fundamental form, $\widetilde{f}$ and $f_{\pi/2}$
have the same Laplacian. Furthermore according to the definition
of the conjugate minimal immersion, $\widetilde{f}$ and $f_{\pi/2}$
have the same square norm of the second fundamental form, i.e.
$|A|^2=(\beta_{11}^2+2\beta_{12}^2+\beta_{22}^2)/E^2$, where we
used the notations in Definition
\ref{def:conjugate minimal immersion}.

By \eqref{eq:Jacobi operator I} and Theorem \ref{thm:FCS80},
the proof of the lemma is complete.
\end{proof}

\section{Stability of helicoids}

For hyperbolic and parabolic catenoids in $\HH^3$ defined in
$\S$\ref{subsub:hyperbolic catenoild-lorentz}
and $\S$\ref{subsub:parabolic catenoild-lorentz},
do Carmo and Dajczer proved that they are
globally stable (see \cite[Theorem (5.5)]{dCD83}). Furthermore,
Candel proved that the hyperbolic and parabolic catenoids are
least area minimal surfaces (see \cite[p. 3574]{Can07}).

The following result can be found in \cite[Theorem (3.31)]{dCD83}.

\begin{theorem}[do Carmo-Dajczer]
\label{thm:conjugate minimal surfaces}
Let $f:\Mscr^2\to\HH^3$ be a minimal catenoid defined in
$\S${}\ref{subsec:catenoids in hyperboloid model}.
Its conjugate minimal surface is the geodesically-ruled minimal surface
$\Hcal_a$ given by \eqref{eq:helicoid in hyperboloid model} where
\begin{equation}\label{eq:relations for CMS}
   \begin{cases}
      a=\sqrt{(\atilde+1/2)/(\atilde-1/2)}\ ,
             &\text{if}\ \Mscr^2=\Mscr_{1}^{2}(\atilde)\
              \text{is spherical}\ ,\\
      a=\sqrt{(\atilde-1/2)/(\atilde+1/2)}\ ,
             &\text{if}\ \Mscr^2=\Mscr_{-1}^{2}(\atilde)\
              \text{is hyperbolic}\ ,\\
      a=1\ , &\text{if}\ \Mscr^2=\Mscr_{0}^{2}\
              \text{is parabolic}\ .
   \end{cases}
\end{equation}
\end{theorem}

Now we are able to prove the main theorem.

\begin{theorem}\label{thm:main theorem}
For a family of minimal helicoids $\{\Hcal_a\}_{a\geq{}0}$
in the hyperbolic $3$-space $\HH^3$ given by
\eqref{eq:helicoid in hyperboloid model},
there exist a constant $a_c\approx{}2.17966$ such that
the following statements are true:
  \begin{enumerate}
    \item $\Hcal_a$ is a globally stable minimal surface
          if $0\leq{}a\leq{}a_c$, and
    \item $\Hcal_a$ is an unstable minimal surface with index one
          if $a>a_c$.
  \end{enumerate}
\end{theorem}

\begin{proof}[{\bf Proof of Theorem \ref{thm:main theorem}}]
When $a=0$, $\Hcal_a$ is a hyperbolic plane, so it is
globally stable.

According to Theorem~\ref{thm:conjugate minimal surfaces},
when $0<a<1$, $\Hcal_a$ is conjugate to
the hyperbolic catenoid $\Mscr_{-1}^{2}(\atilde)$, where
$a=\sqrt{(\atilde-1/2)/(\atilde+1/2)}$ by
\eqref{eq:relations for CMS}, and when $a=1$,
$\Hcal_a$ is conjugate to
the parabolic catenoid $\Mscr_{0}^{2}$ in $\HH^3$.
Therefore when $0<a\leq{}1$, the helicoid
$\Hcal_{a}$ is globally stable
by \cite[Theorem (5.5)]{dCD83}.

When $a>1$, $\Hcal_a$ is conjugate to the spherical catenoid
$\Mscr_{1}^{2}(\atilde)$ in $\HH^3$ by
Theorem~\ref{thm:conjugate minimal surfaces},
which is isometric to the spherical catenoid
$\CC_{\abar}$ in $\B^3$, where
$2\atilde=\cosh(2\abar)$
by \eqref{eq:relation between catenoids} and
$a=\sqrt{(\atilde+1/2)/(\atilde-1/2)}$ by
\eqref{eq:relations for CMS}.
Therefore $\Hcal_a$ is conjugate to
the spherical catenoid $\CC_{\abar}$ in $\B^3$, where
\begin{equation*}
   a=\coth(\abar)\ .
\end{equation*}
By \cite[Theorem 3.9]{Wang2012}, $\CC_{\abar}$ is globally
stable if $\abar\geq\abar_{c}\approx{}0.49577$, therefore
$\Hcal_a$ is globally stable when
$1<a\leq{}a_c=\coth(\abar_{c})\approx{}2.17966$.

On the other hand,
if $0<\abar<\abar_{c}$, then the spherical
catenoid $\CC_{\abar}$ is unstable with index $1$ (see
\cite[Theorem 1.2]{Wang2012}), therefore
$\Hcal_a$ is unstable with index $1$ when
$a>a_c$.
\end{proof}

\bibliographystyle{amsplain}
\bibliography{ref_helicoid}
\end{document}